\documentclass[12pt]{amsart}
\usepackage{amsthm, amssymb}
\usepackage{amsfonts, color, amscd,mathabx}
\usepackage{epsfig,multicol}
\usepackage{xypic}
\usepackage{tikz}
\theoremstyle{plain} \numberwithin{equation}{section}
\newtheorem{thm}{Theorem}[section]

\newtheorem{prop}[thm]{Proposition}

\theoremstyle{definition}
\newtheorem{defn}{Definition}
\newtheorem{exam}[thm]{Example}
\theoremstyle{remark}
 \newtheorem{rem}{Remark}

\def\R{\Bbb R}
\def\F{\Bbb F}
\def\w{\widetilde}
\def\V{\text{Vert}}
\def\rank{\text{Rank}}

\begin{document}
\title[ convex hulls of face-vertex incident vectors of 3-colorable polytopes]{\large \bf  convex hulls of face-vertex incident vectors of 3-colorable Polytopes}
\author[Bo Chen]{Bo Chen}
\author[Chen Peng]{Chen Peng}
\author[Yueshan Xiong]{Yueshan Xiong}
\thanks{Supported in part by grants from National Natural Science Foundation of China(No. 11371093, No. 11801186) }
\address{School of Mathematics and Statistics, Huazhong University of Science and Technology, Wuhan, 430074, P. R.  China}
\email{bobchen@hust.edu.cn}

\address{School of Mathematics and Statistics, Huazhong University of Science and Technology, Wuhan, 430074, P. R.  China}
\email{775461719@qq.com}

\address{School of Mathematics and Statistics, Huazhong University of Science and Technology, Wuhan, 430074, P. R.  China}
\email{yueshan\_xiong@hust.edu.cn}

 \begin{abstract} The convex hulls of face-vertex incident vectors of  3-face-colorable convex polytopes are computed. It's found that every such convex hull is a $d$-polytope with $d+2$ or $d+3$ vertices. Utilizing Gale transform and Gale diagram, we calculate its combinatorial structure. Finally, a necessary and sufficient criterion  for combinatorial equivalence of two such convex hulls is given.
 \end{abstract}

\maketitle

\section{Introduction}\label{int}

In this paper, we focus on a class of 3-face-colorable convex polytope. A convex polytope $P$ is \emph{3-face-colorable}(\emph{3-colorable} in abbreviation) if each facet of $P$ can be assigned a number(color) among 1, 2, and 3, such that adjoin facets are equipped with pairwise different colors. Hence a 3-colorable convex polytope must be 3-dimensional and simple(each vertex belongs to exact three faces).

There is a conjecture proposed by D. W. Barnette in 1969, which says that the 1-skeleton of any 3-colorable polytope is Hamiltonian, i.e., such graph has a cycle passing through each vertex once. The conjecture is still open nowadays.

In this paper, we always assume that $P$ is a 3-colorable polytope with $2n$ vertices(3-colorable polytopes must have even number of vertices). Applying the Euler's formula, we know that $P$ has $2n$ vertices, $3n$ edges and $n+2$ faces, i.e. the f-vector of $P$ is $(2n,3n,n+2).$

Let $V(P)$ be the set of vertices of $P$. For any face $f$ of $P$, set 
\begin{equation*}
\begin{aligned}
\xi_f: &~ V(P) & \rightarrow & ~\{0,1\} \\
      & ~~~~~v&\mapsto  & \begin{cases} 1, & \text{ if } v\in f;\\ 0, & \text{ otherwise}.  \end{cases}
\end{aligned}
\end{equation*}
 
Now  choose an order on $V(P)$, and identify $V(P)$ by $\{1,2,\cdots, 2n\}(|V|=2n)$. Then $\xi_f$ could be identified by a $2n$-entries vector. $\xi_f$ is called the \emph{face-vertex incident vector} of $P$ corresponding to face $f$.
Let $S$ be the set of all face-vertex incident vectors of $P$, i.e., $$S=S(P)=\{\xi_f ~|~ f \text{ is a face of } P \}\subseteq \R^{2n}.$$ 

As pointed out in \cite{cly}, when  $S$ is treated as a set of vectors in $\F_2^{2n}$ where $\F_2$ is the two-element field,  $S$ spans a binary self-dual codes with minimal distance 4. The convex hull of binary self-dual codes was investigated in \cite{ka}.

In this paper we investigate the convex hull $$\w{P}=conv(S)\subseteq \R^{2n}.$$

Theorem \ref{dim} gives the dimension of the convex hull $\w{P}$ is $n$ or $n-1$. Obviously, the vectices set is coincide with $S$, hence $\w{P}$ has exact $n+2$ vertices. So $\w{P}$ is a convex polytope with few vertices: the difference between the number of vertices and dimension is 2 or 3.

Heavily utilizing   the theory of Gale  transform and Gale diagram(\cite{g}), we compute the combinator of $\w{P}$ in Proposition \ref{1}, \ref{2}, \ref{3} and \ref{4}. In detail, let $m_i=m_i(P)$ be the number of faces of $P$ colored by $i=1,2,3.$ Assume $m_1\leq m_2\leq m_3.$ Then

 \begin{itemize}
\item[(1)] if 
$m_1< m_2<m_3$, then $\w{P}$ is the convex hull of a point $w$ and an $n$-simplex where the point $w$ is beyond exact $m_2-1$ facets of the $n$-simplex.
\item[(2)]  if 
 $m_1< m_2=m_3$, then $\w{P}$ is an $m_1$-fold $n$-pyramid with basis $C(2m_2,2m_2-2)$, a $(2m_2-2)$-dimensional cyclic polytope with $2m_2$ vertices.
\item[(3)]  if 
$m_1= m_2<m_3$, then $\w{P}$ is an $m_3$-fold $n$-pyramid with basis $C(2m_2,2m_2-2)$.
\item[(4)]  if 
$m_1= m_2=m_3$, then $\w{P}$ is the convex hull of a point $w$ and an $(m_2-1)$-fold $(n-1)$-pyramid with basis  $C(2m_2,2m_2-2)$, where $w$ is exactly beyond all the non-simplex facet of the $(m_2-1)$-fold $(n-1)$-pyramid.
\end{itemize}

Finally, we point out in theorem \ref{iso} that two convex hulls $\w{P}$ and $\w{P'}$ are combinatorially equivalent if and only $P$ and $P'$ have same f-vectors, same  type of the coloring  vectors $(m_1,m_2,m_3)$ and $m_2(P)=m_2(P')$.



\section{ $\w{P}$ is a polytope with few vertices}



\subsection{Vertices of $\w{P}$}
Since $\xi_f$'s are all $(0,1)$-vector, the vertex set $V$($=V(\w{P})$) of their convex hull $\w{P}$ is coincide with $S=\{\xi_f ~|~ f\text{ is a face of }P \}$. Hence $|V|=f_2(P)=n+2$.  
Set $$
V_i=V_i(P)=\{\xi_f ~|~ f \text{ is a face of }P \text{ colored by } i\}, i=1,2,3.
$$
Then $V=V_1\cup V_2\cup V_3$. A vertex of $\w{P}$ is \emph{colored by $i$}, if it belongs to $V_i$. 
Denote by $$m_i=|V_i |$$
the number of vertices of $\w{P}$ colored by $i=1,2,3$.
 Hence $$m_1+m_2+m_3=n+2.$$

\subsection{Dimension of $\w{P}$}
\begin{thm}\label{dim}
$\dim \w{P} =\begin{cases} n-1, &\text{ if } m_1=m_2=m_3;\\ n, &\text{ otherwise}.\end{cases}$
\end{thm}

\begin{proof}
Set 
$$A=
\begin{bmatrix}
1& 0 &\cdots& 0&-1\\
0&1&\cdots& 0&-1\\
\vdots&&\ddots &\vdots&-1\\
0&&\cdots&1&-1
\end{bmatrix}_{(n+1)\times (n+2)}
$$

$$
A\begin{bmatrix}
\xi_{f_1}\\
\vdots\\
\xi_{f_{n+2}}
\end{bmatrix}= 
\begin{bmatrix}
\xi_{f_1}-\xi_{f_{n+2}}\\
\vdots\\
\xi_{f_{n+1}}-\xi_{f_{n+2}}
\end{bmatrix}
$$

Then $\dim \w{P} = \rank_{\R}
\begin{bmatrix}
\xi_{f_1}-\xi_{f_{n+2}}\\
\vdots\\
\xi_{f_{n+1}}-\xi_{f_{n+2}}
\end{bmatrix} = \rank_{\R} (A\begin{bmatrix}
\xi_{f_1}\\
\vdots\\
\xi_{f_{n+2}}
\end{bmatrix} )$.

As an analogy of corollary 5.4 in  \cite{cly} , $\text{Rank}_{\R} S = n.$
Applying the Sylvester rank inequality, 
$$ (n+1)+n-(n+2) \leq \dim \w{P}  \leq \min\{ n+1, n \},$$
i.e., 
$$ n-1 \leq \dim \w{P}  \leq n.$$

Now we claim that $ \dim \w{P}=n-1$ if and only if $m_1=m_2=m_3$.


Notice that $\sum_{f\in V_i} \xi_f =\underline{1} , i=1,2,3,$ where $\underline{1}$ is the $2n$-entries vector whose entries are all 1(\cite{cly} Proposition 5.6 (2)).
Let $\w{\xi_f}$ be the vector obtained by adding another entry 1 at the end of $\xi_f$.
It's easy to check that $\rank \{ \w{\xi_f} ~|~ f \} =\dim \w{P} +1.$ Hence

$$ \sum_{f\in V_i} \w{\xi_f}= (\underline{1}, m_i), i=1,2,3.$$

Choose arbitrarily two vectors $v$ and $w$ in $V_2$ and $V_3$ respectively. It's easy to check that 
$\{ \xi_f ~|~ f\}-\{v,w\}$ are linearly independent and $v, w$ can be uniquely written as linear combinations:
$v=\sum_{f\in V_1} \xi_f -\sum_{g\in V_2-\{v\}} \xi_g, 
w=\sum_{f\in V_1} \xi_f -\sum_{g\in V_3-\{w\}} \xi_g.$

$\dim \w{P}=n-1 \Leftrightarrow \rank \{ \w{\xi_f} ~|~ f\} =n \Leftrightarrow \text{ both } \w{v} \text{ and } \w{w}$  can be uniquely  written as linear combinations of $\{ \w{\xi_f} ~|~ f\}-\{\w{v}, \w{w}\}  \Leftrightarrow  
\w{v} = \sum_{f\in V_1} \w{\xi_f} -\sum_{g\in V_2-\{v\}} \w{\xi_g}=(v, m_1- m_2+1), 
\w{w}=\sum_{f\in V_1} \w{\xi_f} -\sum_{g\in V_3-\{w\}} \w{\xi_g} =(w, m_1-m_3+1)
 \Leftrightarrow m_1=m_2=m_3.$ 
\end{proof}







\begin{rem} 

$\w{P}$ is a convex polytope of $n+2$ vertices and of dimension $n$ or $n-1$.
\end{rem}

\section{ Gale transform and Gale diagram}

Gale transform and Gale diagram is a effective tool to investigate the combinatorics of polytope with few vertices. The reader is referred  to \cite{g} for the details.

In this section we give some key definitions and results.
The notation here on Gale transform and Gale diagram follows \cite{g}. 

\begin{defn}
Let $V=\{v_1,\cdots, v_n\}$ be a set of (column)vectors in an Euclidean space and
 $\{\beta_1, \cdots, \beta_{n-d-1}\}$ be a basis of the null space of matrix 
$$
\begin{bmatrix} 
1& \cdots & 1 \\
 v_1& \cdots &v_n 
\end{bmatrix},
$$ where $d=\dim \text{conv } V$(i.e., the rank of the matrix minus one).

Denoted by $\bar{v}_j\in \R^{n-d-1}$ the $j$-th row of the matrix $\begin{bmatrix} \beta_1,& \cdots, & \beta_{n-d}\end{bmatrix}$, $ j=1,2,\cdots, n.$ 
The multi-set $\overline{V}=\{\bar{v}_i ~|~ i=1,\cdots, n\}$ is called a \emph{Gale transform} of $V=\{v_1, \cdots, v_n\}$. 

 Normalizing elements in $\overline{V}$ onto the unit sphere $S^{n-d-2}\subseteq \R^{n-d-1},$ we get a \emph{Gale diagram} of $V$.

\end{defn}

 Notation:
  (a) Let $J$ be a subset of $\{1,2,\cdots n\}$. We shall write $V(J)=\{v_j ~|~ j\in J\}$ and $\overline{V(J)}=\{ \bar{v}_j ~|~ j\in J\}$.
  
  (b) Let $A$ be a subset of an Euclidean space. Denote by \emph{relint}$(A)$ the set of relative interior points of $A$.
  
  
  \begin{defn}
  Let $Q$ be a convex polytope with vertex set $V=\{v_1,\cdots,v_n\}$ and $J\subseteq \{1,2,\cdots,n\}$. $V(J)$ is called a \emph{coface} of $Q$ if $V(J^c)$ is the vertices set of a face of $Q$, where $J^c$ is the complement of $J$.
  \end{defn}
 
\begin{thm}[\cite{g} p.88]\label{gd}
 Let $Q$ be a convex polytope with vertex set $V=\{v_1,\cdots,v_n\}$ and $\emptyset\neq J\subseteq \{1,2,\cdots,n\}$.  
Then  $V(J)$ is a coface of $Q$ if and only if 
$0\in \text{relint}( \text{conv }\overline{V(J)})$
\end{thm}

\begin{defn}
A convex polytope $Q$ is \emph{simplicial} if all the proper faces of $Q$ are simplexes.

A convex polytope $Q$ is $k$-neighborly if  any choose of $k$ vertices of $Q$ is exactly the vertex set of a proper face of $Q$. 

A $d$-\emph{pyramid} $Q$ is the convex hull of the union of a $(d-1)$-polytope $K$(\emph{basis} of the pyramid) and a point $A$(\emph{apex} of the pyramid), where the point $A$ does not belong to the affine space aff$ (K)$ of $K$.
\end{defn}

\begin{thm}[\cite{g} p.88 and \cite{gs} Theorem 20]\label{ct}
Let $V$ be the set of vertices of  a convex $d$-polytope $Q$. Then
\begin{itemize}
\item[(1)] $Q$ is simplicial if and only if 
$0\notin \text{relint}( \text{conv}(\overline{V}\cap H)$ for every hyperplane $H$ in $\R^{n-d-1}$ containing $0$.
\item[(2)]  $Q$ is a pyramid if and only if $0\in \overline{V}$. Furthermore, if $0\in \overline{V}$ then $P$ is a pyramid at every $v_i\in V$ such that $\bar{v}_i=0$.
\end{itemize}
\end{thm}

\begin{exam} 
The Gale diagram of the cyclic polytope $C(2k,2k-2)$ is 

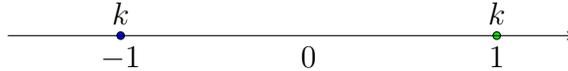
\begin{figure}[!ht]
\centering
\begin{tikzpicture}[scale=0.5]
\path (5,0) coordinate (P3);
\path (10,0) coordinate (P1);
\path (15,0) coordinate (P2);

\draw[fill=green](P2) circle [radius=0.1];
\draw[fill=blue](P3) circle [radius=0.1];

\draw[->] (2,0)--(17,0);

\node[below] at (10,0) {$0$};
\node[above] at (P2) {$k$};
\node[below] at (P2) {$1$};
\node[above] at (P3) {$k$};
\node[below] at (P3) {$-1$};
\end{tikzpicture}
\end{figure}
since the cyclic polytope is simplicial and $(k-1)$-neighborly.

\end{exam}

\section{Combinatorics of $\w{P}$}

For results of  polytopes with few vertices, the reader is referred to  \cite{g}. 



\begin{prop}\label{1}
Suppose $m_3> m_2>m_1\geq 2$. Then the Gale diagram of $\w{P}$ is 

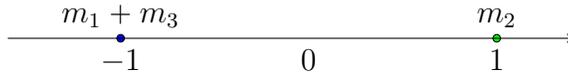
\begin{figure}[!ht]
\centering
\begin{tikzpicture}[scale=0.5]
\path (5,0) coordinate (P3);
\path (15,0) coordinate (P2);

\draw[fill=green](P2) circle [radius=0.1];
\draw[fill=blue](P3) circle [radius=0.1];

\draw[->] (2,0)--(17,0);

\node[below] at (10,0) {$0$};
\node[above] at (P2) {$m_2$};
\node[below] at (P2) {$1$};
\node[above] at (P3) {$m_1+m_3$};
\node[below] at (P3) {$-1$};
\end{tikzpicture}
\end{figure}
And the vertices of $\w{P}$ colored by 1 or 3 corresponding to point -1, and the vertices colored by 2 are corresponding to point 1. Hence 
\begin{itemize}
\item[(1)] conv $V(J)$ is a proper face of $\w{P}$ if and only if $V_2\not\subseteq V(J)$ and $V_1\cup V_3 \not\subseteq V(J)$;
\item[(2)] $\w{P}$ is a simplicial convex polytope.
\end{itemize}
Furthermore, 
\begin{itemize}
\item[(3)] $\w{P} \cong T_{m_2-1}^n:=\text{conv}(v,\Delta^n)$, where $\Delta^n$ is an $n$-simplex, $v\notin \Delta^n$ and $v$ is beyond $m_2-1$ facets of $\Delta^n$.
\end{itemize}
\end{prop}

\begin{proof}
By the theorem \ref{dim}, $\dim \w{P} =n  $ and $\w{P}$ has $n+2$ vertices. The Gale transform of $\V(\w{P})=S$ lies in a line.

It's easy to show that a Gale transform is 
$\{1-k,\cdots,1-k, k,\cdots, k, -1,\cdots,-1\}$
where $1-k, k$ and $-1$ are corresponding to vertices in $\V(\w{P})$  colored by $1,2 $ and $3$ respectively and $k=\frac{m_3-m_1}{m_2-m_1}>1$. Hence the multiplicities of $1-k$, $k$ and $-1$ are $m_1, m_2, m_3$ respectively.

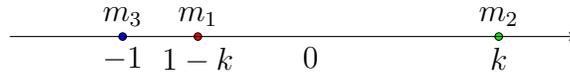
\begin{figure}[!ht]
\centering
\begin{tikzpicture}[scale=0.5]
\path (5,0) coordinate (P3);
\path (7,0) coordinate (P1);
\path (15,0) coordinate (P2);

\draw[fill=red](P1) circle [radius=0.1];
\draw[fill=green](P2) circle [radius=0.1];
\draw[fill=blue](P3) circle [radius=0.1];

\draw[->] (2,0)--(17,0);

\node[below] at (10,0) {$0$};
\node[above] at (P1) {$m_1$};
\node[below] at (P1) {$1-k$};
\node[above] at (P2) {$m_2$};
\node[below] at (P2) {$k$};
\node[above] at (P3) {$m_3$};
\node[below] at (P3) {$-1$};
\end{tikzpicture}
\end{figure}

(1) Using the criterion in theorem \ref{gd}, we know that $V(J^c)$ is not a proper coface of $P\Leftrightarrow 0\notin \text{relint conv } \overline{V(J^c)} \Leftrightarrow$ either $ V_2\subseteq J^c$ or $V_1\cup V_3 \subseteq J^c$. Thus (1) follows.

(2) By the thereom \ref{ct} (1), we get that $\w{P}$ is a simplicial convex polytope.

(3) Since $\w{P}$ is simplicial $n$-polytope with $n+2$ vertices, $\w{P}$ must combinatorial equivalent to $T_k^n$ for some $k\in \{1,2, \cdots, [{n\over 2}]\}$.

Choose a vertex $v_0\in V_2$, it easy to check that conv$(V\setminus \{v_0\})$ is an $n$-simplex since $|V\setminus\{v_0\}|=n+1$ and affine dimension of $V\setminus\{v_0\}$ is $n$. Denote the $n$-simplex conv$(V\setminus \{v_0\})$ by $\Delta^n$. Obviously, $v_0\notin \Delta^n$.

Since $\{v_0,v\}$ is not a coface of $\w{P}$ for any $v\in V_2\setminus \{v_0\}$, i.e., conv$(V(\w{P})-\{v_0,v\})$ is a facet of $\Delta^n$ but not a face of $\w{P}$, $v_0$ must be beyond the $n-1$ simplex of $\Delta^n$ which does not contain (opposite to) $v$. On the other side, $\{v_0, w\}$ is a coface of $\w{P}$ for any vertex $w$ of $\w{P}$ colored by 1 or 3. Hence $v_0$ is not beyond such corresponding facet of $\Delta^n$.
Therefore, $v_0$ is beyond exact $m_2-1$ facet of $\Delta^n$, i.e., 
 $\w{P} \cong T_{m_2-1}^n$.

\end{proof}


\begin{prop}\label{2}
Suppose $m_3=m_2>m_1\geq 2$.  Then the Gale diagram of $\w{P}$ is 

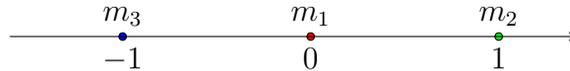
\begin{figure}[!ht]
\centering
\begin{tikzpicture}[scale=0.5]
\path (5,0) coordinate (P3);
\path (10,0) coordinate (P1);
\path (15,0) coordinate (P2);

\draw[fill=red](P1) circle [radius=0.1];
\draw[fill=green](P2) circle [radius=0.1];
\draw[fill=blue](P3) circle [radius=0.1];

\draw[->] (2,0)--(17,0);

\node[below] at (10,0) {$0$};
\node[above] at (P1) {$m_1$};
\node[above] at (P2) {$m_2$};
\node[below] at (P2) {$1$};
\node[above] at (P3) {$m_3$};
\node[below] at (P3) {$-1$};
\end{tikzpicture}
\end{figure}

Hence
\begin{itemize}
\item[(1)] Conv$(J)(J\subsetneq S)$ is a proper face of $\w{P}$ if and only if either $V_2, V_3\not\subseteq J$ or $V_2\cup V_3 \subseteq J$;
\item[(2)] $\w{P}$ is NOT a simplicial convex polytope;
\end{itemize}
Furthermore, 
\begin{itemize}
\item[(3)] $\w{P}\cong T_{m_2-1}^{n,m_1}$ is an $m_1$-fold $n$-pyramid  with $C(2m_2,2m_2-2)$, which is a $(2m_2-2)$-dimensional cyclic polytope  with $2m_2$ vertices.
\end{itemize}
\end{prop}

\begin{proof}

By the theorem \ref{dim}, $\dim \w{P} =n  $ and $\w{P}$ has $n+2$ vertices. The Gale transform of $\V(\w{P})=S$ lies in a line.

The Gale transform is 
$\{0,\cdots,0, 1,\cdots, 1,  -1,\cdots,-1\}$, 
where $0, 1$ and $-1$ are corresponding to vertices in $\V(\w{P})$  colored by $1,2 $ and $3$ respectively 
Hence the multiplicities of $0$, 
$1$ and $-1$ are $m_1, m_2, m_3$ respectively.

(1) Using the criterion in theorem \ref{gd}, we know that conv$(J^c)$ is a proper coface of $P\Leftrightarrow 0\in \text{relint conv } \overline{V(J^c)} \Leftrightarrow$ either $J^c\subseteq V_1$ or   $  J^c \cap V_j\not= \emptyset, j=2,3$. Thus (1) follows.

(2) By the thereom \ref{ct} (1), we get that $\w{P}$ is not a simplicial convex polytope.




(3) Consider the convex hull $Q$ of vertices of $\w{P}$ colored by 2 and 3. Then $Q$ is a face of $\w{P}$ by (1). Hence the Gale diagram of $Q$ is

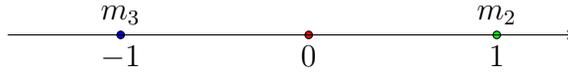
\begin{figure}[!ht]
\centering
\begin{tikzpicture}[scale=0.5]
\path (5,0) coordinate (P3);
\path (10,0) coordinate (P1);
\path (15,0) coordinate (P2);

\draw[fill=red](P1) circle [radius=0.1];
\draw[fill=green](P2) circle [radius=0.1];
\draw[fill=blue](P3) circle [radius=0.1];

\draw[->] (2,0)--(17,0);

\node[below] at (10,0) {$0$};
\node[above] at (P2) {$m_2$};
\node[below] at (P2) {$1$};
\node[above] at (P3) {$m_3$};
\node[below] at (P3) {$-1$};
\end{tikzpicture}
\end{figure}
So $\dim Q=2m_2-2$ and $Q\cong C(2m_2,2m_2-2).$

By the theorem \ref{ct} (2), $\w{P}$ is an $m_1$-fold $n$-pyramid with basis $Q$.
\end{proof}

Similarly, we can prove
\begin{prop}\label{3}
Suppose $m_3> m_2=m_1\geq 2$. Then the Gale diagram of $\w{P}$ is 
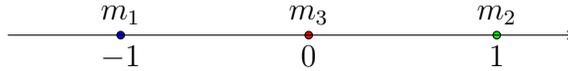
\begin{figure}[!ht]
\centering
\begin{tikzpicture}[scale=0.5]
\path (5,0) coordinate (P3);
\path (10,0) coordinate (P1);
\path (15,0) coordinate (P2);

\draw[fill=red](P1) circle [radius=0.1];
\draw[fill=green](P2) circle [radius=0.1];
\draw[fill=blue](P3) circle [radius=0.1];

\draw[->] (2,0)--(17,0);

\node[below] at (10,0) {$0$};
\node[above] at (P1) {$m_3$};
\node[above] at (P2) {$m_2$};
\node[below] at (P2) {$1$};
\node[above] at (P3) {$m_1$};
\node[below] at (P3) {$-1$};
\end{tikzpicture}
\end{figure}

Hence
\begin{itemize}
\item[(1)] $J\subsetneq S$ is a proper face of $\w{P}$ if and only if either $V_1, V_2\not\subseteq J$ or $V_1\cup V_2\subseteq J$;
\item[(2)] $\w{P}$ is NOT a simplicial convex polytope;
\item[(3)] $\w{P} \cong T_{m_2-1}^{n,m_3}$ is an $m_3$-fold $n$-pyramid  with $C(2m_2,2m_2-2)$.
\end{itemize}
\end{prop}


\begin{prop}\label{4}
Suppose $m_3=m_2=m_1\geq 2$. Then the Gale diagram of $\w{P}$ is 
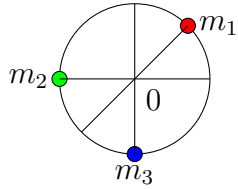
\begin{figure}[!ht]
\centering
\begin{tikzpicture}
\path (0,-1) coordinate (P3);
\path (0.707,0.707) coordinate (P1);
\path (-1,0) coordinate (P2);

\draw (0,0) circle[radius=1];

\draw[fill=red](P1) circle [radius=0.1];
\draw[fill=green](P2) circle [radius=0.1];
\draw[fill=blue](P3) circle [radius=0.1];

\draw[-] (P2)--(1,0);
\draw[-] (P3)--(0,1);
\draw[-] (P1)--(-0.707,-0.707);

\node[below right] at (0,0) {$0$};
\node[right] at (P1) {$m_1$};
\node[left] at (P2) {$m_2$};
\node[below] at (P3) {$m_3$};
\end{tikzpicture}
\end{figure}

\begin{itemize}
\item[(1)] $J\subsetneq S$ is a proper face of $\w{P}$ if and only if  $V_i\nsubseteq J, i=1,2,3$;
\item[(2)] $\w{P}$ is a simplicial convex polytope and $(m_2-1)$-neigborly;
\item[(3)] $\w{P}$ is a convex hull of a point $w$ and a convex $(n-1)$-polytope $Q$, where $Q$ is an $(m_2-1)$-fold $(n-1)$-pyramid  with basis $C(2m_2,2m_2-2)$ and the point $w$ is just beyond exact $m_2-1$ non-simplex facet of $Q$.
\end{itemize}
\end{prop}

\begin{proof}

By the theorem \ref{dim}, $\dim \w{P} =n-1 $ and $\w{P}$ has $n+2$ vertices. The Gale transform of $\V(\w{P})=S$ lies in the  plane:
\begin{equation*}
\{(1,1)^T, \cdots, (1,1)^T,  (-1,0)^T, \cdots,  (-1,0)^T, (0,-1)^T, \cdots, (0,-1)^T \}
\end{equation*}
where $(1,1)^T, (-1,0)^T$ and $(0,-1)^T$ are corresponding to vertices colored by 1, 2 and 3 respectively.

(1) Conv$(J^c)$ is a proper coface of $\w{P}$ $\Leftrightarrow 0\in \text{relint conv } \overline{V(J^c)} \Leftrightarrow V_i\cap J^c\not=\emptyset, i=1,2,3.\Leftrightarrow V_i \not\subseteq J, i=1,2,3.$



(2) By the thereom \ref{ct} (2) and the Gale diagram, $\w{P}$ is simplicial. 
By the Gale diagram, any subset of $V(\w{P})$ with size $m_2-1$ forms a face of $\w{P}$. Hence $\w{P}$ is $(m_2-1)$-neigborly.

(3) Consider the convex hull $Q$ of all vertices of $\w{P}$ except one, $w$,  colored by 1. A Gale transform of $Q$ is
\begin{equation*}
\{\overbrace{1,\cdots,1}^{m_2}, \overbrace{-1, \cdots, -1}^{m_3}, \overbrace{0,\cdots,0}^{m_1-1}\}
\end{equation*} 
where $1, -1$ and $0$ are corresponding to vertices colored by 2, 3 and 1 respectively. Hence the Gale diagram of $Q$ is 

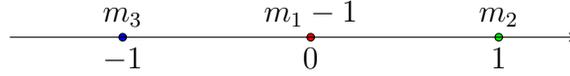
\begin{figure}[!ht]
\centering
\begin{tikzpicture}[scale=0.5]
\path (5,0) coordinate (P3);
\path (10,0) coordinate (P1);
\path (15,0) coordinate (P2);

\draw[fill=red](P1) circle [radius=0.1];
\draw[fill=green](P2) circle [radius=0.1];
\draw[fill=blue](P3) circle [radius=0.1];

\draw[->] (2,0)--(17,0);

\node[below] at (10,0) {$0$};
\node[above] at (P1) {$m_1-1$};
\node[above] at (P2) {$m_2$};
\node[below] at (P2) {$1$};
\node[above] at (P3) {$m_3$};
\node[below] at (P3) {$-1$};
\end{tikzpicture}
\end{figure}
Compare to the Gale diagram in proposition \ref{2}, $Q\cong T_{m_2-1}^{n-1, m_1-1}$ is a $(m_1-1)$-fold $(n-1)$-pyramid  with basis $C(2m_2,2m_2-2)$.

Since $\w{P}$ is simplicial and $\w{P}=conv(Q, w)$, $w$ must be beyond  all the non-simplex facets of $Q$. By the Proposition \ref{2}, $J\subsetneq S-\{w\}$ is a facet of $Q$ if and only if either $J=S-\{w,v_1,v_2\}$ or $J=S-\{w,w'\}$, where the color of $v_1,v_2$ and $w'(\ne w)$ are 1, 2 and 3 respectively. Obviously,  facets whose vertex sets has form $S-\{w,v_1,v_2\}$($S-\{w,w'\}$ respectively) are (not, respectively) simplexes. Hence there are exactly $m_1-1$ non-simplex facet of $Q$, each of which is corresponding to a vertices(distinct to $w$) colored by 1.

On the other side, all facet of $Q$ with type $S-\{w,v_1,v_2\}$ is also a facet of $\w{P}$. Hence $w$ is not beyond such facet of $Q$. Thus $w$ is exactly beyond all the non-simplex face of $Q$.
\end{proof}





\section{Combinatorial equivalence of $\w{P}$ }
Denote by $m(P)=(m_1(P),m_2(P),m_3(P))$ where $m_1(P)\leq m_2(P)\leq m_3(P)$. We say that $m(P)$ has type I, II, III, or IV, if $m(P)$ satisfies the condition in Proposition \ref{1}, \ref{2}, \ref{3} or \ref{4} respectively.

\begin{thm}\label{iso}
Let $P$ and $P'$ be two 3-dimensional convex 3-colorable polytopes. Then $\w{P}$ and $\w{P'}$ are combinatorially equivalent if and only if
\begin{itemize}
\item[(1)] $P$ and $P'$ has same number of faces( or vertices, or edges, equivalently);
\item[(2)] $m(P)$ and $m(P')$ has same type;
\item[(3)] $m_2(P)=m_2(P')$.
\end{itemize}

\end{thm}

\begin{proof}
It's sufficient to prove the necessarity.
Suppose $\w{P}\cong \w{P'}$. 

(1) Since their differences of number of vertices and dimension must be equal,  both $m(P)$ and $m(P')$ are of type IV or not. Notice that $\dim\w{P}= f_2(P)-3$ if $m(P)$ is of type IV, and $\dim\w{P}= f_2(P)-2$ otherwise, where $f_2(P)$ is the number of faces of $P$. Hence $\dim \w{P} =\dim \w{P'} $ implies that $f_2(P)=f_2(P').$

(2) Suppose both $m(P)$and $m(P')$ are not of type IV, i.e., $\w{P}$ and $\w{P'}$ are both convex $n$-polytopes with $n+2$ vertices.

There is a criterion in \cite{g} p.108, which says that $\w{P} \cong \w{P'}$ if and only if $(\bar{m}_0, \bar{m}_1, \bar{m}_{-1}) =(\bar{m}'_0, \bar{m}'_1, \bar{m}'_{-1}) $ or $(\bar{m}_0, \bar{m}_1, \bar{m}_{-1}) =(\bar{m}'_0, \bar{m}'_{-1}, \bar{m}'_{1}) $, where $\bar{m}_j$'s and $\bar{m}'_j$'s are multiplicities of $j$ in the Gale diagram of $\w{P}$ and $\w{P'}$ respectively, $j=0,\pm 1.$ 

Notice that $m_3\geq m_2\geq m_1$. Comparing the multiplicities of  Gale diagrams in the Proposition \ref{1}, \ref{2} and \ref{3}, we get that $\w{P} \cong \w{P'}$ implies that $(\bar{m}_0, \bar{m}_1, \bar{m}_{-1}) =(\bar{m}'_0, \bar{m}'_1, \bar{m}'_{-1}) $. Hence $m(P)$ and $m(P')$ must be of same type.
Furthermore, $m_2(P)=m_2(P')$.

Of course, when $m(P)$ and $m(P')$ are of type IV, $m_2(P)=m_2(P')$ since $\dim\w{P}=3m_2(P)-3=\dim\w{P'}=3m_2(P')-3$.

\end{proof}


\begin{thebibliography}{99}
\bibitem{cly}
Bo Chen, Zhi L\"u and Li Yu,  \emph{Self-daul binary codes from small covers and simple polytopes}, Algebraic $\&$ Geometric Topology 18(2018) 2729-2767.

\bibitem{ka}Navin Kashyap, \emph{On the convex geometry of binary linear codes}, in Proc. Inaugural UC San Diego Workshop Inf. Theory Appl., La Jolla, CA,
Feb. 2006. Available: http://ita.ucsd.edu/workshop/06/talks

\bibitem{g}
Branko Gr\"unbaum, \emph{Convex Polytopes}, GTM 221, Springer, 2nd edtion.

\bibitem{gs}
B. Gr\"unbaum and G. C. Shephard, \emph{Convex Polytopes}, Bull. London Math. Soc., 1(1969), 257-300.

\bibitem{rs}Richard P. Stanley, \emph{Combinatorics and commutative algebra}, (2nd ed.) Birkhäuser (1996) ISBN 0-81764-369-9
\end{thebibliography}
\end{document}